\documentclass[11pt]{amsart}
\usepackage{amsmath}
\usepackage{graphicx}
\usepackage{epstopdf}
\usepackage{amssymb,amscd,array}
\usepackage{color}
\usepackage{float}

\makeindex \makeatletter
\def\captionof#1#2{{\def\@captype{#1}#2}}
\makeatother

\newcounter{tablegroup}
\newcounter{subtable}[tablegroup]

\newtheorem{thm}{Theorem}[section]
\newtheorem{cor}[thm]{Corollary}
\newtheorem{lem}[thm]{Lemma}

\newtheorem{defn}[thm]{Definition}
\newtheorem{rem}[thm]{\bf Remark}

\numberwithin{equation}{section}

\newcommand{\eps}{\varepsilon}

\begin{document}
\title{Monotone Maps on dendrites and their induced maps}

\author{Haithem Abouda and Issam Naghmouchi}
\address{Haithem Abouda, University of Carthage, Faculty
of Sciences of Bizerte, Department of Mathematics,
Jarzouna, 7021, Tunisia;}

\email{haithem.abouda@yahoo.fr}
\address{Issam Naghmouchi, University of Carthage, Faculty
of Sciences of Bizerte, Department of Mathematics,
Jarzouna, 7021, Tunisia;}

\email{issam.naghmouchi@fsb.rnu.tn;issam.nagh@gmail.com}



\begin{abstract}
A continuum $X$ is a dendrite if it is locally connected and contains no simple closed curve, a self mapping $f$ of $X$ is called monotone if the preimage of any connected subset of $X$ is connected. If $X$ is a dendrite and $f:X\to X$ is a monotone continuous map then we prove that any $\omega$-limit set is approximated by periodic orbits and the family of all $\omega$-limit sets is closed with respect to the Hausdorff metric.  Second, we prove that the equality between the closure of the set of periodic points, the set of regularly recurrent points and the union of all $\omega$-limit sets holds for the induced maps $\mathcal{F}_n(f):\mathcal{F}_n(X)\to \mathcal{F}_n(X)$ and $\mathcal{T}_n(f):\mathcal{T}_n(X)\to \mathcal{T}_n(X)$ where $\mathcal{F}_n(X)$ denotes the family of finite subsets of $X$ with at most $n$ points, $\mathcal{T}_n(X)$ denotes the family of subtrees of $X$ with at most $n$ endpoints and $\mathcal{F}_n(f)=2^f_{\mid\mathcal{F}_n(X)}$,  $\mathcal{T}_n(f)=2^f_{\mid\mathcal{T}_n(X)}$, in particular there is no Li-Yorke pair for these maps. However, we will show that this rigidity in general is not exhibited by the induced map $\mathcal{C}(f):\mathcal{C}(X)\to \mathcal{C}(X)$ where $\mathcal{C}(X)$ denotes the family of sub-continua of $X$ and $\mathcal{C}(f)=2^f_{\mid\mathcal{C}(X)}$, we will discuss an example of a homeomorphism $g$ on a dendrite $S$ which is dynamically simple whereas its induced map $\mathcal{C}(g)$ is $\omega$-chaotic and has infinite topological entropy.
\end{abstract}
\maketitle

\section{\bf Introduction}

Let $\mathbb{Z}_{+}$ and $\mathbb{N}$ be the sets of non-negative integers and positive integers respectively. Let $X$ be a compact metric space with metric $d$ and $f: X\longrightarrow X$ be a continuous map. Denote by $f^{n}$ the $n$-th iterate of $f$; that is,  $f^{0} =\textrm{id}_{X}$: the Identity and $f^{n} = f\circ f^{n-1}$ if $n\geq 1$. For any $x\in X$ the subset $O_{f}(x)=\{f^{n}(x): \ n\in\mathbb{Z}_{+}\}$ is called the $f$-orbit of $x$. A point $x\in X$  is called \textit{periodic} if $f^n(x)=x$ for some $n>0$. When $n=1$ we say that $x$ is a \textit{fixed point}.
The period of $p$ is the least natural number $m$ such that $f^m(p)=p$.
A subset $A$ of $X$ is called \textit{$f$-invariant} (resp. strongly $f$-invariant) if
$f(A)\subset A\  (\textrm{resp}.\ f(A)=A$). It is called \textit{a minimal set of $f$} if it is non-empty, closed, $f$-invariant and minimal (in the sense of inclusion) for these properties.  Note that a nonempty closed set $M \subseteq X$ is minimal if and only if the orbit of every point from $M$ is dense in $M$. The
 $\omega$-limit set of a point $x$ is the set
\begin{align*}
 \omega_{f}(x) & = \{y\in X: \exists\ n_{i}\in \mathbb{N},
n_{i}\rightarrow\infty, \lim_{i\rightarrow\infty}d(f^{n_{i}}(x), y)
= 0\}\\
& = \underset{n\in \mathbb{N}}\cap\overline{\{f^{k}(x): k\geq n\}}.
\end{align*}

The set $\omega_{f}(x)$ is a non-empty, closed and strongly
invariant set. If $\omega_f(x)$ is finite then it is a periodic orbit. If $\omega_{f^m}(x)$ is finite for some $m\in \mathbb{N}$ then $\omega_f(x)$ is also finite (see \cite{lsB} for more details).\\
 A point $x\in X$ is said to be:

-  \textit{recurrent} if for every open set $U$ containing $x$ there exists $n\in \mathbb{N}$ such that $f^n(x)\in U$ (equivalently, $x\in \omega_{f}(x)$).

- \textit{regularly recurrent} if for any $\varepsilon>0$, there is
$N\in\mathbb{N}$ such that $d(x,f^{kN}(x))<\varepsilon$ for all
$k\in\mathbb{N}$.

It is known that if $x$ is regularly recurrent then
 $\omega_{f}(x)$ is a minimal set (see
\cite{lsB}, Proposition 5, Chapter V).

Let  $\textrm{Fix}(f)$, $\textrm{\textrm{P}(f)}$, $\textrm{RR}(f)$, $\textrm{R}(f)$ and $\Lambda(f)$ denote the set of fixed
points, periodic points, regularly recurrent, recurrent points and the union of all $\omega$-limit sets
respectively. Then we have the following inclusion relation
$\textrm{Fix}(f)\subset \textrm{P}(f)\subset \textrm{RR}(f)\subset \textrm{R}(f) \subset \Lambda(f)$.

\medskip

The notion of \textit{topological entropy} for a continuous self-map of a compact metric space  was first introduced by Adler, Konheim and McAndrew in \cite{Adler}. We recall here the equivalent definition formulated by Bowen \cite{Bowen}, and independently by Dinaburg \cite{Dinaburg}:
Let $(X,d)$ be a compact metric space and let $f:X\rightarrow X$ be a map. Fix $n\geq 1$ and $\varepsilon> 0$. A subset $E$ of $X$ is called $(n,f,\varepsilon)-$separated if for every two different points $x,y\in E$, there exists $0\leq j< n$ with $d(f^j(x),f^j(y))> \eps$. Denote by $sep(n,f,\varepsilon)$ the maximal possible cardinality of an $(n,f,\varepsilon)$-separated set in X. Then the topological entropy of $f$ is defined by $$h(f)=\underset{\varepsilon\to 0}\lim \underset{n\to+\infty}\limsup \frac{1}{n}\ Log\ sep(n,f,\varepsilon).$$

A pair $(x,y)\in X\times X$ is called \textit{proximal} if $\underset{n\to \infty}\liminf\ d(f^n(x),f^n(y))=0$. If $\underset{n\to \infty}\limsup\ d(f^n(x),f^n(y))=0, (x,y)$ is called \textit{asymptotic}. A pair $(x,y)$ is called a \textit{Li-Yorke pair} (of $f$) if it is proximal but not asymptotic. A subset $S$ of $X$ containing at least two points is called a \textit{Li-Yorke scrambled set} (of $f$) if for any $x,y\in S$ with $x\neq y, (x,y)$ is a Li-Yorke pair. We say that $f$ is \textit{Li-Yorke chaotic} if there exists an uncountable scrambled set of $f$. It is proved in \cite{Blanchard} that if $h(f)>0$ then $f$ is Li-Yorke chaotic.\\

A subset $S$ of $X$ containing at least two points is called an \textit{$\omega$-scrambled} set for $f$ if for any $x,y\in S$ with $x\neq y$ the following conditions are fulfilled:
 \begin{itemize}
   \item[(i)] $\omega_f(x)\setminus \omega_f(y)$ is uncountable,
   \item[(ii)] $\omega_f(x)\cap \omega_f(y)\neq\emptyset$,
   \item[(iii)] $\omega_f(x)\setminus \textrm{P}(f)\neq\emptyset.$
 \end{itemize}
The map $f$ is called \textit{$\omega$-chaotic}  if there is an uncountable $\omega$-scrambled set (\cite{Li}). Following Theorem 1.4 and Corollary 2.6 from \cite{Marzougui} the definition of $\omega$-scrambled set is reduced only to conditions (i) and (ii) when the space $X$ is either a completely regular continuum (i.e every non-degenerate sub-continuum of $X$ has non-empty interior) or a zero-dimensional compact space .\\
It is noteworthy that $\omega$-chaos is not related to the topological entropy. The paper (\cite{Pikula}, Example 4.3) provides an example of a map which is $\omega$-chaotic with zero topological entropy. Surely, positive topological entropy is not enough to imply $\omega$-chaos, since there is known example of minimal map with positive topological entropy (\cite{Rees}).\\

 A continuum is a compact connected metric space. A topological space
is arcwise connected if any two of its points can be joined by an arc. We use the terminologies from Nadler \cite{Nadler}. An arc is any space homeomorphic to the compact interval $[0,1]$. By a \textit{dendrite} $X$, we mean a locally connected continuum which contains no homeomorphic copy to a circle. Every sub-continuum of a dendrite is a dendrite (\cite{Nadler}, Theorem 10.10) and every connected subset of $X$ is arcwise connected (\cite{Nadler}, Proposition 10.9). As dendrites are locally connected, any connected component of open subset of a dendrite is open in $X$. In addition, any two distinct points $x,y$ of a dendrite $X$ can be joined by a unique arc with endpoints $x$ and
$y$, denote this arc by $[x,y]$ and let denote by $[x,y)=[x,y]\setminus\{y\}$ (resp. $(x,y]=[x,y]\setminus\{x\}$ and $(x,y)=[x,y]\setminus\{x,y\}$).
If $X$ is a dendrite and $p\in X$, then the order of $p$, denoted by $ord_X(p)$,  is the number of connected components  of $X\setminus \{p\}$. If $ord_X(p)=1$, we say that $p$  is an \textit{endpoint} and if $ord_X(p)> 2$ then we say that $p$ is a \textit{branch point}. Denote by $\textrm{End}(X)$ and $\textrm{Br}(X)$ the sets of endpoints and branch points of $X$ respectively. A point $x\in X\setminus \textrm{End}(X)$ is called a \textit{cut point}. Denote by $\textrm{Cut}(X)$ the set of cut points of $X$. It is known that $\textrm{Cut}(X)$ is dense in $X$ (\cite{Kura}, VI, Theorem 8, p.302). A \textit{tree} is a dendrite with finite set of endpoints. A continuous map from a dendrite into itself is called a \textit{dendrite map}. Every dendrite has the fixed point property (see \cite{Nadler}); that is when $f$ is a dendrite map, then $\textrm{Fix}(f)\neq \emptyset$.\ Given a closed subset $S$ of $X$, denote by  $[S]=\displaystyle\cup_{x\neq y\in S}[x,y]$, called the \textit{convex hull} of $S$.\\

\begin{defn}\cite{Kura} Let $X,\ Y$ be two topological spaces. A map $f:X\rightarrow Y$ is said to be monotone of for any connected subset $C$ of\ $Y,\ f^{-1}(C)$ is connected.
\end{defn}

Notice that $f^n$ is monotone for every $n\in \mathbb{N}$ when $f$ itself is monotone. Also note that when $X$ is a dendrite then the monotony of $f:X\to Y$ implies that of $f_{\mid A}:A\to Y$ for any connected subset $A$ of $X$, this is due to the fact that the intersection of any two connected subset of $X$ is connected  (\cite{Nadler}, Theorem 10.10).\\
 Given a continuum $X$ with a metric $d$, we  denote by $2^X$ the hyperspace of all nonempty closed subsets of $X$. For any two subsets $A$ and $B$ of $X$, define $d(A,B)=\underset{x\in A, \ y\in B}\inf d(x,y)$ and denote by $d(x,A)=d(\{x\},A)$. The Hausdorff metric $d_H$ is defined as follows: Let $A,B\in 2^X$
$$d_H(A,B)=\max\big\{\sup_{x\in A}d(x,B),\sup_{y\in B}d(A,y)\big\}.$$
This defines a distance on $2^X$ (\cite{Nadler2}, Theorem 2.2). Further, we denote by $\mathcal{C}(X)$ the hyperspace of all sub-continua of $X$ (i.e, the family of all connected elements of $2^X$). If $X$ is a continuum and $f:X\rightarrow X$ is a map, we can consider the maps (called the induced maps) $2^f:2^X\rightarrow 2^X$ and $\mathcal{C}(f):\mathcal{C}(X)\rightarrow \mathcal{C}(X)$ defined as follows:  $2^f(A)=f(A),\ \text{for\ each}\ A\in 2^X$ and  $\mathcal{C}(f)=2^f_{\mid \mathcal{C}(X)}.$
It is known that the continuity of $f$ implies the continuity of both $2^f$ and $C(f)$ (\cite{Nadler2}, Lemma 13.3).\\
 For a metric space $(X,d)$, the interior, the closure and the diameter of a subset $A\subseteq X$ is denoted by int$(A),\ \overline{A}$ and $\textrm{diam}(A)$ respectively. Let $x\in X,\ Y\subseteq X$ and $\eps> 0$. One write $B(x,\eps)=\{x'\in X :d(x,x')< \eps\}$  for the $\eps$-ball and $B(Y, \eps)=\{x\in X, d(x,Y)< \eps\}$.\\
The natural question arises: what is the connection between dynamical properties of $f$ and those the induced maps $2^f$ and $\mathcal{C}(f)$. For papers related to this topic, see (\cite{Banks},\cite{Acosta},\cite{Matviichuk},\cite{Kwietniak}).

In this paper, we are dealing with monotone dendrite maps $f:X\to X$ and their induced maps $\mathcal{F}_n(f):\mathcal{F}_n(X)\to \mathcal{F}_n(X)$, $\mathcal{T}_n(f):\mathcal{T}_n(X)\to \mathcal{T}_n(X)$ and $\mathcal{C}(f):\mathcal{C}(X)\to \mathcal{C}(X)$, where $\mathcal{F}_n(X)$ denotes the family of finite subsets of $X$ with at most $n$ points, $\mathcal{T}_n(X)$ denotes the family of subtrees of $X$ with at most $n$ endpoints and $\mathcal{F}_n(f)=2^f_{\mid\mathcal{F}_n(X)}$,  $\mathcal{T}_n(f)=2^f_{\mid\mathcal{T}_n(X)}$.\\
 This paper is organized as follow: We give in Section 2, some preliminaries results which are useful in the rest of the paper. In Section 3, we give more description on the dynamic of monotone dendrite maps. In section 4, we prove that any $\omega$-limit set is approximated by periodic orbits and the space of all $\omega$-limit sets is closed with respect to the Hausdorff metric. Sections 5 and 6 are devoted to the study of the dynamics of the induced maps $\mathcal{F}_n(f)$ and $\mathcal{T}_n(f)$ respectively. Finally, in Section 7 we discuss an example of a homeomorphism dendrite $g:S\rightarrow S$ dynamically simple whereas its induced map $\mathcal{C}(g)$ is $\omega$-chaotic and has infinite topological entropy.

\section{\bf Preliminaries}

The proof of the following Lemma is trivial.
\begin{lem}\label{lem0}If $J$ is a compact interval and $f:J\rightarrow J$ is a continuous monotone map, then for any $x\in J,\ \omega_f(x)$ is either a fixed point or a periodic orbit of period $2$.
\end{lem}

\begin{lem}\label{lem8}\rm{(\cite{Ma2}, Lemma 2.1)}\ Let $(X,d)$ be a dendrite. Then for every $\varepsilon > 0$,\ there exists $0<\delta=\delta(\varepsilon)<\eps$ such that, for any $x,y\in X$ with $d(x,y)\leq \delta,$\ we have $\mathrm{diam}([x,y])< \varepsilon$.
\end{lem}

\begin{lem}\rm{(\cite{Ma2}, Lemma 2.2)}\label{lem5} Let $[a,b]$ be an arc in a dendrite $(X,d)$ and $w\in [a,b)$.\ There is $\delta > 0$ such that if $v\in X$ with $d(v,b)\leq \delta$ then $[a,v]\supset [a,w]$.
\end{lem}
\begin{lem}\rm{(\cite{Nagh}, Lemma 3.2)}\label{lem6}
 Let $f:X\rightarrow X$ be a monotone dendrite map. Let $a\in \textrm{Fix}(f)$ and $x\in X$.\ If $\omega_f(x)$ is infinite then for every $n\in \mathbb{N},\ [a,x]\cap [a,f^n(x)]=[a,u_n]$ where $u_n\in \textrm{Fix}(f^n)$.
\end{lem}

\begin{lem}\rm{(\cite{Nagh2}, Lemma 2.2)}\label{lem7}
Let $X$ be a dendrite and $(p_{n})_{n\in\mathbb{N}}$ be a sequence of $X$ such that $p_{n+1}\in(p_{n},p_{n+2})$ for all $n\in\mathbb{N}$, and $\underset{n\to +\infty}\lim p_{n}=p_{\infty}$. Let $U_{n}$ be the connected component of $D\setminus\{p_{n},p_{\infty}\}$ that
contains the open arc $(p_{n},p_{\infty})$. Then $\underset{n\to
+\infty}\lim \mathrm{diam}(U_{n})=0$.
\end{lem}

\begin{lem}\label{lem1}\rm{(\cite{Nagh}, Lemma 2.8)}
Let $(X,d)$ be a dendrite and $f:X\rightarrow X$ a monotone dendrite map. Then for any $x,y\in X,\ f([x,y])=[f(x),f(y)]$.
\end{lem}
\begin{lem}\label{lem-2} Let $(X,d)$ be a dendrite and $f:X\rightarrow X$ a monotone dendrite map. Let $a\in \textrm{Fix}(f)$ and $x\in X$. If for some $n\in \mathbb{N}$, we have either  $f^n(x)\in [a,x]$ or $x\in [a,f^n(x)]$, then $\omega_f(x)$ is finite.
\end{lem}
\begin{proof}
By Lemma \ref{lem1}, $f^n([a,x])=[a,f^n(x)]$. So, if $f^n(x)\in [a,x]$ then by Lemma \ref{lem0}, $\omega_f(x)$ is finite. If $x\in [a,f^n(x)]$. From Lemma 2.10 in \cite{Nagh},  $\omega_{f^n}(x)=\{b\}\subset \textrm{Fix}(f^n)$ and so $\omega_f(x)$ is finite.
\end{proof}
\begin{thm}\label{th1}\rm{(\cite{Nagh}, Theorem 1.2)} Let $f:X\rightarrow X$ be a monotone dendrite map. Then for any $x\in X$, we have:
\begin{itemize}
  \item[$(i)$] $\omega_f(x)$ is a minimal set.
  \item[$(ii)$] $\omega_f(x)\subset \overline{\textrm{P}(f)}$.
\end{itemize}
\end{thm}
\begin{lem}\label{lem17}Let $(X,d)$ be a compact metric space and $f:X\rightarrow X$ a continuous map. If $x\in \textrm{RR}(f)$ and $y\in \textrm{R}(f)$ such that $(x,y)$ is an asymptotic pair then $x=y$.
\end{lem}
\begin{proof}
Suppose that $x\neq y$. Then there exists $\eps > 0$ such that $y\notin \overline{B(x,\eps)}$. As $x\in \textrm{RR}(f)$, there exists $N\in \mathbb{N}$ such that $f^{kN}(x)\in B(x,\frac{\eps}{2}),\ \forall k\in \mathbb{N}$ and by the fact that $(x,y)$ is an asymptotic pair, $f^{kN}(y)\in B(x,\eps)$ for $k$ large enough. So $\omega_{f^N}(y)\subset \overline{B(x,\eps)}$, a contradiction with the fact that  $y\in \textrm{R}(f)=R(f^N)$.

\end{proof}
\section{\bf More emphasis of the dynamic monotone dendrite maps}

\begin{lem}\label{lem11} Let $f:X\rightarrow X$ be a monotone dendrite map and $T$ is a subtree of $X$ with $k$ endpoints. Then we have the following assertions:
\begin{itemize}
\item[(i)] For all $e\in \textrm{End}(f(T))$, there exists $e'\in \textrm{End}(T)$ such that $f(e')=e$.
\item[(ii)] $f(T)$ is a subtree of $X$ with $f(\textrm{End}(T))\supset \textrm{End}(f(T))$.
\end{itemize}
\end{lem}

\begin{proof}
(i) Let $e\in \textrm{End}(f(T))$ and  $e_{-1}\in f^{-1}(\{e\})\cap T$. Take $a,b\in \textrm{End}(T)$ such that $e_{-1}\in [a,b]$. So, $e\in f([a,b])=[f(a),f(b)]\subset f(T)$. Thus, $e$ is either $f(a)$ or $f(b)$.\\
(ii) We have $f(T)$ a sub-dendrite of $X$. By assertion (i), $f(\textrm{End}(T))\supset \textrm{End}(f(T))$ so $\textrm{End}(f(T))$ is finite with at most $k$ points where $k$ is the number of endpoints of $T$.
\end{proof}
\begin{lem}\label{lem4}Let $f:X\rightarrow X$ be a monotone dendrite map. If $M$ is an $\omega$-limit set of $f$ and $D$ is the the convex hull of $M$. Then we have the following assertions.
\begin{itemize}
\item[(i)] $D$ is strongly invariant.
\item[(ii)] $f_{\mid D}:D\to D$ is monotone.
\end{itemize}
\end{lem}
\begin{proof}
(i) Since $D=\displaystyle\cup_{a \neq b \in M}[a,b],$ we have $f(D)=\displaystyle\cup_{a \neq b \in M}f([a,b])$. By Lemma \ref{lem1},  $f(D)=\displaystyle\cup_{a \neq b \in M}[f(a),f(b)]$. Therefore, $f(D)=\displaystyle\cup_{a \neq b \in M}[a,b]$ since $f(M)=M$. Thus, $D$ is strongly invariant.\\
(ii) If $C$ is a connected subset of $D$ then by the fact that $f$ is monotone $f^{-1}(C)$ is connected and so the intersection $f^{-1}(C)\cap D$ is connected as it is the intersection of two connected subsets of a dendrite. Then $f_{\mid D}$ is monotone.
\end{proof}
\begin{lem}\label{lem3}
Let $f:D\rightarrow D$ be a monotone dendrite map and $x\in D$. If $M$ is an $\omega$-limit set of $f$ and $D$ is the convex hull of $M$. Then, we have:
\begin{itemize}
\item[(i)] Any arc in $X$ contains at most two points from $M$.
\item[(ii)] $M=\textrm{End}(D)$.
\end{itemize}
\end{lem}
\begin{proof}
(i) Let $x\in X$. If $M$ is finite then it is a periodic orbit. So if we denote by $T$ the convex hull of $M$ then by Lemma \ref{lem4} (i), $T$ is a subtree of $X$ strongly invariant. By Lemma  \ref{lem11} (ii), $f(\textrm{End}(T))=\textrm{End}(T)$ hence due to the minimality of $M$, $\textrm{End}(T)=M$. So there are no three points from $M$ lying in the same arc. If now $M$ is infinite and $a,b,c\in M$ such that for some arc $I$, $\{a,b,c\}\subset I$ and $b\in (a,c)$. By Theorem \ref{th1}, there are $p,q\in \textrm{P}(f)$ such that $b\in [a,c]\cap [p,q]$. Since $p$ and $q$ are periodic, there is $n\in \mathbb{N}$ such that $f^n(p)=p$ and $f^n(q)=q.$ By Lemma \ref{lem1}, $f^n([p,q]=[p,q])$. So, by Lemma \ref{lem0}, $\omega_{f^n}(b)$ is finite and so is $\omega_f(b)$. But this contradicts that $M=\omega_f(b)$ is an infinite $\omega$-limit set.\\
(ii) It can be deduced easily from assertion (i).
\end{proof}

\begin{thm}\label{th4}Let $(X,d)$ ba a dendrite and $f:X\rightarrow X$ a monotone dendrite map. Then
$$\overline{\textrm{P}(f)}=\Lambda(f)=\textrm{RR}(f)$$
\end{thm}

 \begin{proof} By \rm{(\cite{Nagh}, Theorem 1.4)}, we have $\overline{\textrm{P}(f)}=\textrm{R}(f)=\Lambda(f)$ and it is known that $\textrm{RR}(f)\subseteq \Lambda(f)$. So, it suffices to prove that $\Lambda(f)\subseteq \textrm{RR}(f)$.\ Suppose that $y\in \Lambda(f)$,\ that is for some $x\in X,\ y\in \omega_f(x)$.\ Suppose that $\omega_f(x)$ is infinite. Let $a\in \textrm{Fix}(f)$. First, we prove that $[a,y]$ contains a sequence of periodic points $(p_n)_n$ such that $p_{n+1}\in (p_n,y)$ for each $n$ and $\underset{n\to +\infty}\lim p_n=y$.\ Let $z_1\in (a,y)$.\ Then by Lemma \ref{lem5}, for some $\alpha_1 > 0$,\ if $y'\in B(y,\alpha_1)$ then $[a,y']\supset [a,z_1]$.\ Since $y$ is recurrent, there is $m_1\in \mathbb{N}$ such that $f^{m_1}(y)\in B(y,\alpha_1)$.\ By Lemma \ref{lem6}, we have $[a,f^{m_1}(y)]\cap [a,y]=[a,p_1]$ where $p_1\in \textrm{Fix}(f^{m_1})$ and we have also $y\neq p_1$ (since otherwise by the minimality of $\omega_f(x)$, $\omega_f(x)=\omega_f(y)$ and so $\omega_f(x)$ is finite which is a contradiction). As $f^{m_1}(y)\in B(y,\alpha_1),\ [a,f^{m_1}(y)]\supset [a,z_1]$. Therefore $[a,z_1]\subset [a,f^{m_1}(y)]\cap [a,y]=[a,p_1]$.\ So, $p_1\in [z_1,y)$. Let $z_2\in (p_1,y)$ be such that $\mathrm{diam}([z_2,y])<\frac{d(z_1,y)}{2}$. Let $\alpha_2>0$ be such that for any $y^{'}\in B(y,\alpha_2)$, $[a,y^{'}]\supset [a,z_2]$. By the fact that $y$ is recurrent, there is $m_2\in\mathbb{N}$ such that $f^{m_2}(y)\in B(y,\alpha_2)$. Then similarly as above we prove that $[a,f^{m_2}(y)]\cap [a,y]=[a,p_2]$ where $p_2\in \textrm{Fix}(f^{m_2})\cap [z_2,y)$. \ We proceed the construction of the sequence $(p_n)_n$ by induction in the following way: Suppose that $p_n$ is already defined, let $z_{n+1}\in (p_n,y)$ be such that $\mathrm{diam}([z_{n+1},y])< \frac{d(z_n,y)}{2}$. By Lemmas \ref{lem5} and  \ref{lem6}\ and the fact that $y$ is recurrent, there is $m_{n+1}\in \mathbb{N}$ and $p_{n+1}\in \textrm{Fix}(f^{m_{n+1}})$ such that $[a,f^{m_{n+1}}(y)]\cap [a,y]=[a,p_{n+1}]$ and $p_{n+1}\in [z_{n+1},y)$.

By construction, we have for all $n$, $p_{n+1}\in (p_n,y)$ and $d(p_{n+1},y)<\frac{d(p_1,y)}{2^n}$. \ Therefore, the sequence $(p_n)_n$ converges to $y$. Now, we are going to prove that $y\in \textrm{RR}(f)$:\ We define for each $n\in \mathbb{N},\ V_n\neq \emptyset$ to be the closure of the connected component of $X\setminus \{p_n,y\}$ that contains $(p_n,y)$.\ By Lemma \ref{lem7},  $\underset{n\to
+\infty}\lim \mathrm{diam}(V_{n})=0$.\ Let $\varepsilon> 0$ and let $N\in \mathbb{N}$ be such that diam$(V_N)<\varepsilon$.\ Take $k\in \mathbb{N}$ be such that $f^k(p_N)=p_N$. We prove then that for all $n\in \mathbb{N}$ we have $f^{kn}(y)\in V_N$.\ In fact, for any $n\in \mathbb{N},\ f^{kn}([a,y])=[a,f^{kn}(y)]\supset [a,p_N]$.\ By Lemma \ref{lem-2},\ $[a,y]$ is not included into $[a,f^{kn}(y)]$. So necessary, $f^{kn}(y)\in V_N$.\ This ends the proof.
\end{proof}

\medskip
 \begin{defn}(\textit{Adding machines})  Let $\alpha=(j_1,j_2,...)$ be a sequence of integers, where each $j_i\geq 2.$\ Let $\triangle_\alpha$ denote all sequences $(x_1,x_2,...)$ where $x_i\in \{0,1,...,j_i-1\}$ for each $i.$\ We put a metric $d_\alpha$ on $\triangle_\alpha$ given by
$$d_\alpha((x_1,x_2,...),(y_1,y_2,...))=\displaystyle\sum_{i=1}^{\infty}\frac{\delta(x_i,y_i)}{2^i}$$
Where $\delta(x_i,y_i)=1$ if $x_i\neq y_i$ and $\delta(x_i,y_i)=0$ if $x_i=y_i$. The addition in $\triangle_\alpha$ is defined as follows:
$$(x_1,x_2,...)+(y_1,y_2,...):=(z_1,z_2,...)$$
where $z_1=(x_1+y_1)\ mod\ j_1$ and $z_2=(x_2+y_2+t_1)\ mod\ j_2$, with $t_1=0$ if $x_1+y_1< j_1$ and $t_1=1$ if $x_1+y_1\geq j_1$. So, we carry a one in the second case. Continue adding and carrying in this way for the whole sequence.\\
We define $f_\alpha:\triangle_\alpha\rightarrow \triangle_\alpha$ by
$$f_\alpha(x_1,x_2,...)=(x_1,x_2,...)+(1,0,0,...)$$
We will refer to the map $f_\alpha:\triangle_\alpha\rightarrow \triangle_\alpha$ as the \emph{adding machine map}.
\end{defn}
\medskip
\medskip
\begin{thm}\label{th2}\rm{(\cite{Block}, Corollary 2.5)} Let $f:X\rightarrow X$ be a continuous map of a compact Hausdorff space to itself. There is a sequence $\alpha$ of prime numbers such that $f$ is topologically conjugate to the adding machine $f_\alpha$ if and only if $X$ is an infinite minimal set for $f$ and each point of $X$ is regularly recurrent.
\end{thm}
Recall that $f$ is topologically conjugate to the adding machine $f_{\alpha}$  means that there exists a homeomorphism $h:X\rightarrow \triangle_\alpha$ such that $h\circ f=f_{\alpha}\circ h$.\\

 So we obtain the following description of the dynamical structure of infinite minimal set for monotone dendrite map.

 \medskip
\begin{cor}Let $f:X\rightarrow X$ be a monotone dendrite map and  $M$ is an infinite $\omega$-limit sets. Then there is a sequence $\alpha$ of prime numbers such that $f_{\mid M}$ is topologically conjugate to the adding machine $f_\alpha$.
\end{cor}

\begin{proof}
By Theorem \ref{th4}, $M\subset \textrm{RR}(f)$. Hence by Theorem \ref{th2}, there is a sequence $\alpha$ of prime numbers such that $f_{\mid M}$ is topologically conjugate to the adding machine $f_\alpha$.
\end{proof}

\medskip
 Recall that the density of periodic points of $f$ trivially implies the density of periodic points of $2^f$. The converse is known to be not true, take for instance any adding machine map $f_\alpha$ it has no periodic points but its induced map $2^{f_ \alpha}$ has dense set of periodic points (see \cite{Banks}). However, it was conjectured by M\'endez (\cite{Mendez2}) that the converse holds for any dendrite map. We affirm this conjecture for the subclass of monotone dendrite maps.

\medskip
\begin{thm}\label{thmm}
Let $f:X\rightarrow X$ be a monotone dendrite map. Then the following statements are equivalent:
\begin{itemize}
\item[(a)] $\textrm{P}(f)$ is dense in $X$.
\item[(b)] $X\setminus \textrm{End}(X)\subset \textrm{P}(f)$ and $f$ is a homeomorphism.
\item[(c)] $P(2^f)$ is dense in $2^X$.
\item[(d)] $R(2^f)$ is dense in $2^X$.
\end{itemize}
\end{thm}

\begin{proof}
(a)$\Leftrightarrow$ (b) follows from (\cite{Nagh},Theorem 1.6). (a)$\Rightarrow$ (c) (\cite{Banks}, Lemma 1)  (c)$\Rightarrow$ (d)\ is clear.
Suppose that $R(2^f)$ is dense in $2^X$. Let $V$ be a non-empty open set of $X$ and take $U$ a non-empty open subset of $X$ such that $\overline{U}\subset V$. By (\cite{Nadler}, Theorem 4.5), the set $\Gamma(U)=\{F\subset U: F$ is a closed set of $X\}$ is open in $2^X$. Let $F\in R(2^f)\cap \Gamma(U)$, then for infinitely many $n$, $f^n(F)\in \Gamma(U)$ so $f^n(F)\subset U$. Then for any $x\in F$, $f^n(x)\in U$ for infinitely many $n\in\mathbb{N}$, it follows that $\omega_f(x)\cap \overline{U}\neq\emptyset$. Thus, $\Lambda(f)\cap \overline{U}\neq\emptyset$. So by Theorem \ref{th4}, $\overline{\textrm{P}(f)}\cap V\neq\emptyset$, hence $\textrm{P}(f)\cap V\neq\emptyset$.
\end{proof}
\begin{rem} \rm{Recently, it is proved in ( \cite{Acosta2}, Theorem 7.3) that if $P(2^f)$ is dense in $2^X$ then $\textrm{P}(f)$ is dense in $X$ for a class of continua called almost meshed (which contains in particular the dendrites with closed set of endpoints). By Theorem \ref{thmm}, $\textrm{P}(f)$ this also hold for every monotone map on any dendrite.}

\end{rem}

\section{\bf Spaces of minimal sets of monotone dendrite maps}

We recall the first point map $r_Y:X\rightarrow Y$ for a dendrite $Y$ contained in a dendrite $X$, it is defined by letting $r_Y(x)=x$ if $x\in Y$,\ and by letting $r_Y(x)$ the unique point $r_Y(x)\in Y$ such that $r_Y(x)$ is a point of any arc in $X$ from $x$ to any point of $Y$ if $x\in X\setminus Y$; this map is well defined and continuous, see \rm{(\cite{Nadler}, Lemma 10.25)}. Note that $r_Y$ is constant in every connected component of $X\setminus Y$.

\medskip
\begin{lem}\label{lem10} Let $f:X\to X$ be a monotone dendrite map and $D$ is a strongly $f$-invariant sub-dendrite of $X$. For any minimal set $M$ of $f$ if $M\cap D=\emptyset$ then $r_D(M)$ is a periodic orbit.
\end{lem}
\begin{proof}
Let $M$ be a minimal subset of $X$ disjoint from $D$. Let $a,b\in r_D(M)$ with $a\neq b$. Suppose that $C_a$ (resp. $C_b$) is the connected component of $X\setminus D$ such that $\overline{C_a}\cap D=\{a\}$ (resp. $\overline{C_b}\cap D=\{b\}$). Then $C_a\cap C_b=\emptyset$ and $C_a\cap M\neq\emptyset\neq C_b\cap M$. Let $x\in C_a\cap M$. By the fact that $C_b$ is an open neighborhood of $b$ in $X$ and $M$ is minimal, there is $n\in\mathbb{N}$ such that $f^n(x)\in C_b$. Therefore, $f^n(a)=b$; indeed, we have $f^n([a,x])=[f^n(a),f^n(x)]\ni b$. So the point $b$ has a preimage $\alpha$ under $f^n$ in the arc $[a,x]$ and recall that $f(D)=D$ which implies that $b$ has also a preimage $\gamma$ in $D$ under $f^n$, hence $a\in [\alpha,\gamma]$. By the fact that $f$ is monotone (and so is $f^n$), $a\in [\alpha,\gamma]\subset f^{-n}(\{b\})$ so $f^n(a)=b$.

As the retraction map $r_D$ is continuous and $M$ is compact, $r_D(M)$ is compact. Thus $r_D(M)$ is a minimal orbit, so it is a periodic orbit.
\end{proof}
\begin{lem}\label{lem9} Let $f:X\rightarrow X$ be a monotone dendrite map and $M$ an infinite $\omega$-limit set of $f$. Suppose that $D$ is the convex hull of $M$  and  $p\in D$ is a periodic point with $T$ be the convex hull of its orbit $O_f(p)$. Then we have :
\begin{itemize}
\item[(i)] For any component $C$ of $X \setminus T,\ \ r_{T}(C)\in O_f(p)$.
\item[(ii)] $ord_D(x)$ is finite for each $x\in D$.
\end{itemize}
\end{lem}

\begin{proof}
(i) Let $x\in (X \setminus T)\cap M$ such that $r_T(x)=p$.\ By Lemma \ref{lem1}, $\forall n\in \mathbb{N},\ f^n([p,x])=[f^n(p),f^n(x)]$ and we have $[f^n(p),f^n(x)]\cap T=\{f^n(p)\}$. Since otherwise, there is $z\in (f^n(p),f^n(x)]\cap T$. As, $f^n(T)=T,\ \exists z_{-n}\in T$ such that $f^n(z_{-n})=z.$ Also there is $z_{-n}'\in (p,x]$ such that $f^n(z_{-n}')=z.$ Hence, $f^{-n}(z)\supset [z_{-n}',z_{-n}]\ni p$ (since $f^{-n}(\{z\})$ is connected). So, $f^n(p)=z$, a contradiction. Take $C$ a connected component of $X\setminus T$ then $C\cap M\neq \emptyset$ and by minimality of $M$ there is $n\in \mathbb{N}$ such that $f^n(x)\in C$. So by the fact that $[f^n(p),f^n(x)]\cap T=\{f^n(p)\},\ r_T(f^n(x))=f^n(p)$. Recall that the map $r_T$  is constant on $C$. Hence $r_T(C)=\{f^n(p)\}$.\\
(ii) By assertion (ii) of Lemma \ref{lem3}, $D$ is a dendrite where its set of endpoints $\textrm{End}(D)=M$ is closed then $ord_D(x)$ is finite for all $x\in D$ (see \cite{Charatonic}).
\end{proof}

\begin{lem}\label{hausd1} Let $(X,d)$ be a dendrite with set of endpoints $\textrm{End}(X)$ closed and let $(X_n)_{n\in\mathbb{N}}$
be a sequence of trees in $X$ satisfying the following properties:
\begin{itemize}
 \item[(a)] $X=\overline{\bigcup_{n\in\mathbb{N}} X_n}$,
\item[(b)] $X_n\subset X_{n+1}$, $\forall n\in\mathbb{N}$,
\item[(c)] For any connected component $C$ of $X\setminus X_n$, $r_{X_n}(C)=\{e_n\}\subset \textrm{End}(X_n)$.
\end{itemize}
Then $\underset{n\to+\infty}\lim \textrm{End}(X_n) = \textrm{End}(X)$ (in the Hausdorff metric).
\end{lem}
\begin{proof} Let $\eps>0$. By Lemma \ref{lem8}, there exists $0<\delta<\eps$ such that if $d(x,y)<\delta$
then $\textrm{diam}([x,y])<\eps$. By (a), it is easy to see that $\underset{n\to+\infty}\lim X_n=X$, then there is $n_0>0$ such that
$d_H(X_n,X)<\delta$ for $n\geq
n_0$. Take $n\geq n_0$. Let $e\in \textrm{End}(X)$, we shall prove that $d(e,\textrm{End}(X_n))<\eps$: If $e\in \textrm{End}(X_n)$ then clearly $d(e,\textrm{End}(X_n))=0$.
If not, the point $e$ belongs to some connected
component
of $X\setminus X_n$. Thus by (b), $e_n: = r_{X_n}(e)\in \textrm{End}(X_n)$ and so
$e_n\in [e,x]$ for any $x\in X_n$. Let $x\in X_n$ be such that
$d(e,x)< \delta$. Then, we have $d(e,e_n)\leq \textrm{diam}([e,x])<\eps$ and therefore $d(e,\textrm{End}(X_n))<\eps$.
 Take now $e_n\in \textrm{End}(X_n)$, we shall prove that  $d(e_n,
\textrm{End}(X))<\eps$: If $e_n\in \textrm{End}(X)$ then clearly $d(e_n, \textrm{End}(X))=0$. If $e_n\notin \textrm{End}(X)$,
then let $C$ be a connected component of $X\setminus X_n$ such that
$\overline{C}\cap X_n=\{e_n\}$. Then for any $x\in X_n$ and $e\in C\cap \textrm{End}(X)$, we have $e_n\in [e,x]$. So
take any point $e\in C\cap
\textrm{End}(X)$ and let $x\in X_n$ be such that $d(e,x)<\delta$. Then $\textrm{diam}
([e,x])<\eps$ and so $d(e,e_n)<\eps$. It follows that $d(e_n,
\textrm{End}(X))<\eps$. In result, we conclude that $d_H(\textrm{End}(X_n),\textrm{End}(X))<\eps$ for any $n\geq n_0$ and therefore $\underset{n\to+\infty}\lim d_H(\textrm{End}(X_n),\textrm{End}(X))=0$. \end{proof}

\begin{thm}\label{th3}Let $(X,d)$ be a dendrite and $f:X\rightarrow X$ a monotone dendrite map. Then any $\omega$-limit set $M$ of $f$ is contained in the Hausdorff closure of periodic orbits in its convex hull.
\end{thm}

\begin{proof}
Note that $M$ is itself a periodic orbit when it is finite, then in this case there is nothing to prove. So suppose that $M$ is infinite. Let $D$ be the convex hull of $M$. Fix $a\in M=\textrm{End}(D)$ and let $\tau=\textrm{diam} (D)$, we are going to construct a sequence of trees $(T_n)_{n}$ in $D$ such that for each $n\in\mathbb{N}$, $\textrm{End}(T_n)$ is a periodic orbit, $T_n\subset T_{n+1}$, $d(a,T_n)\leq \frac{\tau}{n}$ and $\textrm{Cut} (D)=\cup_{n\in\mathbb{N}} T_n$.  Let $T_1=\{\gamma\}$ where $\gamma\in \textrm{Fix}(f)\cap \textrm{Cut}(D)$. The other trees will be obtained by induction as follow: Suppose that $T_n$ was already defined with the required properties mentioned above. Let $C$ be the connected component of $D\setminus T_n$ that contains $a$ and let $C^{'}$ be another connected component of $D\setminus T_n$ that does not contain $a$, hence $(C\cap M)\neq\emptyset\neq (C^{'}\cap M)$ and by assertion (i) of Lemma \ref{lem9}, any arc joining one point from $C$ to another one from $C^{'}$ has non-empty intersection with $\textrm{End}(T_n)$. By the minimality of $M$, there is $s\in\mathbb{N}$ for which $f^s(a)\in C^{'}$. Choose  $0<\delta<\frac{\tau}{n+1}$ such that $f^s(B(a,\delta)\cap D)\subset C^{'}$, set $V=B(a,\delta)\cap D$. By  Lemma \ref{lem4} and Theorem \ref{th1}, there is $q_n\in V\cap \textrm{P}(f)$. Hence, $[q_n,f^s(q_n)]\cap \textrm{End}(T_n)\neq \emptyset$. We let then $T_{n+1}=[O_f(q_n)]$, so we get $T_n\subset T_{n+1}$ and $d(a,T_{n+1})<\frac{\tau}{n+1}$. Now lets prove that $\textrm{Cut} (D)=\cup_{n\in\mathbb{N}} T_n$. Take $x\in \textrm{Cut}(D)$ then for some $b\in M$, $x\in [a,b]$. Let $\eps>0$ be such that if $z\in B(b,\eps)$ then $x\in[a,z]$. Again by the minimality of $M$, for some $k\in \mathbb{N}$, $f^k(a)\in B(b,\eps)$ and by the continuity of $f^k$, there is $\mu>0$ such that $f^k(B(a,\mu))\subset B(b,\eps)$.
Suppose that $Q$ is the connected component of $D\setminus \{x\}$ that contains $a$. Then $Q\cap B(a,\mu)$ is an open neighborhood of $a$ in $D$. Hence, there is $n\in\mathbb{N}$ such that for some $p_n\in Q\cap B(a,\mu)$, thus $[p_n,f^k(p_n)]\ni x$, it follows that $x\in T_n$. Now by assertion (i) of Lemma \ref{lem9}, the sequence of trees $(T_n)_n$ fulfilled condition (c) in Lemma \ref{hausd1}.  And by construction, the sequence $(T_n)_n$ fulfilled either condition (a) and (b) of Lemma \ref{hausd1}, hence we get $\underset{n\to +\infty}\lim d_H(\textrm{End}(T_n),M)=0$.

\end{proof}

\medskip
It is known that, if $X$ is a dendrite then the family of all $\omega$-limit sets of a continuous map $f:X\rightarrow X$  endowed with the Hausdorff metric need not to be compact (\cite{kocan}, Theorem 2). In the case of monotone dendrite map, the family of $\omega$-limit sets coincide with the family of minimal sets and we show that it is a compact subset of $2^X$.

\begin{thm}\label{thm2}
Let $(X,d)$ be a dendrite and $f:X\rightarrow X$ a monotone dendrite map. Then the set of all minimal sets endowed with the Hausdorff metric is compact.
\end{thm}
\begin{proof}
Let $(M_n)_{n}$ be a sequence of minimal sets that converges in the Hausdorff metric to $M$. Then by the compactness of $(2^X,d_H)$, $M$ is a compact set. By Theorem \ref{th3}, one may assume that all $M_n$ are periodic orbits. By the continuity of $2^f, f(M)=M$. Let $D$ be the convex hull of $M$. Thus, $D=\displaystyle\cup_{a\neq b\in M}[a,b]$. Therefore, $f(D)=\displaystyle\cup_{a\neq b\in M}[f(a),f(b)]=D$ since $f(M)=M$.  As $M$ is in the closure of periodic orbits in the Hausdorff metric, we get that $M\subset \overline{\textrm{P}(f)}$ and from Theorem \ref{th4}, we have  $\overline{\textrm{P}(f)}=\textrm{RR}(f)$ thus $M\subset \textrm{RR}(f)$ and so $M$ is the union of minimal sets for $f$. We are going  to show that $M$ is in fact a unique minimal set for $f$. Suppose that $M$ is not minimal, then there exists two minimal sets $K_1\neq\ K_2$ such that $K_1\cup K_2\subseteq M$. Denote by $H_1=[K_1],\ H_2=[K_2]$ and $H=[K_1\cup K_2]$. Denote by $K=r_H(M)$, by the continuity of $r_H$, we have $\underset{n\rightarrow +\infty}\lim r_H(M_n)=r_H(M)=K$. Denote by $N_n=r_H(M_n), \forall n\in \mathbb{N}$. We distinguish two cases, both of them lead us to a contradiction.

\textit{Case 1:} $H_1\subset H_2$ or $H_2\subset H_1$.\\
Suppose that $H_1\subset H_2$, then $H=H_2$ and $K_1\subset \textrm{Cut}(H)$. Now, as $\textrm{End}(H)=K_2$ is closed, each point in $H$ has a finite order in $H$ and the set of branch point of $H$ is discrete (see \cite{Charatonic}). Take a point $c\in K_1$, then $c$ is a Cut point of $H$. Suppose that $r=ord_H(c)$. Then there is a neighborhood $T_r$ of $c$ in $H$ such that $T_r\cap K_2=\emptyset$ and $T_r=\displaystyle\cup_{i=1}^{r}[c,c_i]$ where $[c,c_i]\cap [c,c_j]=\{c\}$ for any $i\neq j$. For each $i\in \{1,...,r\}$, let $C_i$ be the connected component of $H\setminus (T_r\setminus \{c_i\})$ that contains $c_i$ and let $b_i\in C_i\cap K_2.$ Let $\varepsilon> 0$ be such that $B(b_i,\varepsilon)\cap H\subset C_i$ for all $i\in\{1,...,r\}$ and $B(c,\eps)\cap H\subset T_r$. As $\underset{n\to+\infty}\lim d_H(N_n,K)=0$, there exists $n\in \mathbb{N}$ such that $d_H(N_n,K)<\varepsilon$. Thus, for all $i\in \{1,\dots,r\}$ there exist $b_i(n)\in N_n,\ c(n)\in N_n$ such that $d(b_i(n),b_i)< \varepsilon$ and $d(c(n),c)< \varepsilon$.\ Recall that $H$ is strongly invariant as it is the convex hull of the minimal subset $K_2$ (assertion (i) of Lemma \ref{lem4}) and so $N_n=r_{H}(M_n)$  is a periodic orbit by Lemma \ref{lem10}. For some $i\in \{1,\dots,r\}$, we have $c(n)\in [c,c_i]$. Take $j\neq i$, then $c(n)\in [b_i(n),b_j(n)]$. Recall that the three points $c(n),b_i(n)$ and $b_j(n)$ belong to  $N_n$ which is a periodic orbit, a contradiction with Lemma \ref{lem3}.

\textit{Case 2:} $H_2\nsubseteq H_1$ and $H_1\nsubseteq H_2$.\\
In this case, $\textrm{End} (H)=K_1\cup K_2$ and it can be proved easily that $H$ is strongly invariant and so $N_n=r_{H}(M_n)$  is a periodic orbit by Lemma \ref{lem10}. As $K_2\cap H_1=\emptyset$, the number $\eta=d(K_2,H_1)>0$. So for any closed subset $A$ of $H_1$, $d(A,K_2)\geq\eta$. Take any point $z\in K_1$, then for some $\delta>0$, $(B(z,\delta)\cap H)\subset H_1$. For $n$ large enough, we have $d_H(N_n,K)<\delta$ hence $N_n\cap B(z,\delta)\neq\emptyset$ (since $z\in K_1\subset K$) and so $N_n\cap H_1\neq\emptyset$. Recall that $N_n$ is a periodic orbit, thus $N_n\subset H_1$ and so $d(N_n,K_2)>\eta$ for $n$ large enough. However as $\underset{n\to +\infty}\lim d_H(N_n,K)=0$ and $K_2\subset K$, for $n$ large enough $d(N_n,K_2)<\eta$, a contradiction.
\end{proof}

\section{\bf Induced maps on the hyperspaces $\mathcal{F}_n(X)$ }

Given a dendrite $X$ and $n\in\mathbb{N}$. Let $\mathcal{F}_n(X)=\{A\in 2^X: A\ \text{has\ at\ most}\ n\ \text{points}\}$ and denote by $\mathcal{F}_n(f)=2^f_{\mid \mathcal{F}_n(X)}$.

\medskip
\begin{lem}\label{lem13}Let $f:X\rightarrow X$ be a monotone dendrite map. Then for any $x\in X$ there exists $y\in \omega_f(x)$ such that $(x,y)$ is an asymptotic pair.
\end{lem}
\begin{proof}
If $\omega_f(x)$ is finite then Lemma \ref{lem13} is obviously holds. When $\omega_f(x)$ is infinite, see the proof of Theorem 1.2 in \cite{Nagh} .
\end{proof}
\begin{lem}\label{lem14}Let $f:X\to X$ be a monotone dendrite map. Let $n\in \mathbb{N}$ and $A=\{x_1,x_2,...,x_n\}\in \mathcal{F}_n(X)$. Then there exists $B=\{y_1,\dots,y_n\}\in \textrm{RR}(\mathcal{F}_n(f))$ asymptotic to $A$ and such that $y_i\in \omega_f(x_i)$ for any $i\in\{1,\dots,n\}$.
\end{lem}
\begin{proof}
 Let $n\in \mathbb{N}$ and $A=\{x_1,x_2,...,x_n\}\in \mathcal{F}_n(X)$. Then by Lemma \ref{lem13}, for each $1\leq i\leq n$, there exists $y_i\in \omega_f(x_i)$ such that $(x_i,y_i)$ is an asymptotic pair for $f$. Denote by $B=\{y_1,y_2,...,y_n\},$ then $(A,B)$ is an asymptotic pair for $\mathcal{F}_n(f)$. Recall that $\Lambda(f)=\textrm{RR}(f)=\overline{\textrm{P}(f)}$ (Theorem \ref{th4}). So by (\cite{Mendez}, Theorem 5), $B\in \textrm{RR}(\mathcal{F}_n(f))$.
\end{proof}
\begin{thm}\label{prop1} Let $(X,d)$  be a dendrite and  $f:X\rightarrow X$  a monotone dendrite map. Then for any $A\in \mathcal{F}_n(X)$, there exists $B\in \textrm{RR}(\mathcal{F}_n(f))\cap \overline{\textrm{P}(\mathcal{F}_n(f))}$ asymptotic to $A$.
\end{thm}

\begin{proof}
 Let $A=\{x_1,x_2,...,x_n\}\in \mathcal{F}_n(X)$. Then by Lemma \ref{lem14}, there exists $B=\{y_1,y_2,...,y_n\}\in \textrm{RR}(\mathcal{F}_n(f))$ such that $(A,B)$ is an asymptotic pair for $\mathcal{F}_n(f)$ and $y_i\in \omega_f(x_i)$ for any $i\in\{1,\dots,n\}$. By Theorem \ref{th4}, for each $i\in\{1,\dots,n\}$, there is a periodic point $p_i$ of $f$ in the ball $B(y_i,\eps)$, hence the set $P=\{p_1,\dots,p_n\}$ is a periodic point of $\mathcal{F}_n(f)$ and $d_H(B,P)<\eps$. It follows that $B\in \overline{\textrm{P}(\mathcal{F}_n(f))}$.
\end{proof}
\begin{cor}\label{cor1}Let $f:X\rightarrow X$ be a monotone dendrite map. For each $n\in \mathbb{N}$, we have:
\begin{itemize}
\item[(a)] $\omega_{\mathcal{F}_n(f)}(A)$ is minimal for each $A\in \mathcal{F}_n(X)$.
\item[(b)] $\overline{\textrm{P}(\mathcal{F}_n(f))}=\Lambda(\mathcal{F}_n(f))=\textrm{RR}(\mathcal{F}_n(f))$.
\end{itemize}
\end{cor}

\begin{proof}
$(a)$ It follows immediately from Theorem \ref{prop1}.\\
$(b)$ By Lemma \ref{lem17} and Theorem \ref{prop1}, we have $\Lambda(\mathcal{F}_n(f))=\textrm{RR}(\mathcal{F}_n(f))\subset \overline{\textrm{P}(\mathcal{F}_n(f))}$. So it suffices to prove that $\overline{\textrm{P}(\mathcal{F}_n(f))}\subset \textrm{RR}(\mathcal{F}_n(f))$. Let $P_n$ be a sequence of periodic points of $\mathcal{F}_n(X)$ that converges to $P=\{a_1,...,a_k\}$. Clearly, for any $i\in \{1,...,k\}, a_i\in \overline{\textrm{P}(f)}$. Recall that $\overline{\textrm{P}(f)}=\textrm{RR}(f)$ thus $P\subset \textrm{RR}(f)$. By (\cite{Mendez}, Theorem 5), $P\in \textrm{RR}(\mathcal{F}_n(f))$.
\end{proof}
The following Corollary follows immediately from Corollary \ref{cor1} and Theorem \ref{th2}.
\begin{cor}Let $f:X\rightarrow X$ be a monotone dendrite map. Then for any $A\in \mathcal{F}_n(X),\  \omega_{\mathcal{F}_n(f)}(A)$ is either finite or a minimal Cantor set.
\end{cor}
We also deduce easily from Theorem \ref{prop1} the following:
\begin{cor} The induced map $\mathcal{F}_n(f)$ generated by a monotone dendrite map has no Li-Yorke pair. In particular, it has zero topological entropy.
\end{cor}

\section{\bf Induced maps on the hyperspaces $\mathcal{T}_n(X)$}
Given a dendrite $X$ and $n\in \mathbb{N}$, we denote by\ $\mathcal{T}_n(X)$\  the\ family\ of\ trees\ in\ $X$ with at most $n$ endpoints. We show  that $\mathcal{T}_n(X)$ is a closed subsets in $2^X$ (see Lemma \ref{lem15}). From Lemma \ref{lem11}, $f(\mathcal{T}_n(X))\subset \mathcal{T}_n(X)$. So we may study the dynamic of the map $\mathcal{T}_n(f)=2^f_{\mid{\mathcal{T}_n(X)}}$.

\medskip
\begin{lem}\label{lem15} Let $(X,d)$ be a dendrite. Then $\mathcal{T}_n(X)$ is closed in $(2^X,d_H)$.
\end{lem}
\begin{proof}
Let $k\in \mathbb{N}$ and $(T_n)_n$ be a sequence of trees with $k$ endpoints. If $(T_n)_n$ converges in the Hausdorff metric to $T$, then we will prove that $T$ is a tree with at most $k$ endpoints. Indeed, by the fact that $(\mathcal{C}(X),d_H)$ is compact, $T$ is a sub-dendrite of $X$. Suppose that $card(\textrm{End}(T))> k$. Let $\{e_1,...,e_{k+1}\}\subset \textrm{End}(T)$ with $e_i\neq e_j$ for each $i\neq j$. There is $\eps> 0$ such that for any $i\in \{1,...,k+1\}$ and for any $b_i\in B(e_i,\eps), [\{b_1,...,b_{k+1}\}]$ is a subtree of $X$ with $(k+1)$ endpoints. As, $\underset{n\to +\infty}\lim d_H(T_n,T)=0$ so for some $n\in \mathbb{N}$ and for any $i\in \{1,...,k+1\}$, there is $b_i\in T_n\cap B(e_i,\eps)$. Thus, $[\{b_1,...,b_{k+1}\}]$ is a subtree of $X$ with $(k+1)$ endpoints. As, $\{b_1,...,b_{k+1}\}\subset T_n$ then $[\{b_1,...,b_{k+1}\}]\subset T_n$. So, $T_n$ has more then $k$ endpoints, a contradiction.
\end{proof}
\begin{lem}\label{lem12}Let $(X,d)$ be a dendrite. Let $\varepsilon > 0$  and $0<\delta< \varepsilon$ be as in Lemma \ref{lem8}. Let $k\in \mathbb{N}$ and let $A=\{a_1,a_2,...,a_k\}, $ $B=\{b_1,b_2,...,b_k\}$ be two subsets of $X$ with  $d(a_i,b_i)<\delta$ for each $i\in \{1,...,k\}$ then
$d_H([A],[B])<\varepsilon$.
\end{lem}
\begin{proof}
Let $x\in [A].$ If $x\in [B]$ then $d(x,[B])=0$. If $x \notin [B]$ and $x\in A$ then $x=a_i$ for some $i\in \{1,...,k\}$ and so $d(x,b_i)<\delta<\varepsilon$. Therefore, $d(x,[B])<\varepsilon$. If $x \notin ([B] \cup A)$ then for some $i,j\in\{1,...,k\},\  x\in[a_i,a_j]$ where $a_i\neq x\neq a_j$. Let $C_i$ (resp. $C_j$) be the connected component of $X\setminus \{x\}$ that contain $a_i$ (resp. $a_j$), then at least one of them is disjoint from the connected component of $X\setminus\{x\}$ that contain $[B]$, suppose for example $C_i\cap [B]=\emptyset$. We have then $x\in [r_{[B]}(x),a_i]$. Hence, $x\in [b_i,a_i]$.
By Lemma \ref{lem8}, $\textrm{diam}[b_i,a_i]< \varepsilon$ and so $d(x,b_i)< \varepsilon$.
 In conclusion, $d(x,[B])<\varepsilon$. Similarly, we prove that for each $y\in [B],\ d(y,[A])< \varepsilon$. In result, $d_H([A],[B])<\varepsilon$.
\end{proof}
\begin{thm}\label{prop2}Let $(X,d)$ be a dendrite and $f:X\rightarrow X$ a monotone dendrite map. Then for any  $T\in \mathcal{T}_n(X)$, there exists $K\in \textrm{RR}(\mathcal{T}_n(f))\cap\overline{\textrm{P}(\mathcal{T}_n(f))}$ asymptotic to $T$.

\end{thm}

\begin{proof}

Suppose that $T\in \mathcal{T}_n(X) $ and $\textrm{End}(T)=\{a_1,...,a_k\}$ where $1\leq k\leq n$. By Lemma \ref{lem13}, for each $i\in \{1,...,k\}$, there exists $b_i\in \omega_f(a_i)$ asymptotic to $a_i$. Set $K=[\{b_1,...,b_k\}]$. Let $\varepsilon> 0$ and $0<\delta<\varepsilon$ be as in Lemma \ref{lem8}, then there is $N_0\in \mathbb{N}$ such that for any $n\geq N_0, d(f^n(a_i),f^n(b_i))<\delta$, for each $i\in \{1,...,k\}$. By assertion (b) in Lemma \ref{lem11}, $[\{f^n(a_1),...,f^n(a_k)\}]=f^n(T)$ and $[\{f^n(b_1),...,f^n(b_k)\}]=f^n(K).$ By Lemma \ref{lem12}, $d_H(f^n(T),f^n(K))<\varepsilon$, for any $n\geq N_0$. Lets prove now that $K$ is a regularly recurrent point for the induced map $\mathcal{T}_n(f)$ and belongs to the closure of the set of its periodic points. Indeed, the set of endpoints of $K,\ \{c_1,...,c_r\}\subset \{b_1,...,b_k\}\subset \Lambda(f)=\textrm{RR}(f)$ by Theorem \ref{th4}. Let $\varepsilon > 0,\ 0<\delta<\varepsilon$ be as in Lemma \ref{lem8}, there is $N\in \mathbb{N}$ for which $d(c_i,f^{kN}(c_i))<\delta$ for each $i\in \{1,..,r\}$ and for any $k\in \mathbb{N}$. Also by Theorem \ref{th4}, there is a periodic point $p_i\in B(c_i,\delta)$ for each $i\in \{1,..,r\}$, it is clear that the convex hull $P$ of $\{p_1,\dots,p_r\}$ is a tree which belong to $\textrm{P}(\mathcal{T}_n(f))$. By assertion (b) of Lemma \ref{lem11}, $[\{f^{kN}(c_1),...,f^{kN}(c_r)\}]=f^{kN}(K)$ and by Lemma \ref{lem12}, $d_H(K,f^{kN}(K))<\varepsilon$. Again by Lemma \ref{lem12}, $d_H(K,P)<\eps$. This implies  that $K\in \textrm{RR}(\mathcal{T}_n(f))\cap \overline{\textrm{P}(\mathcal{T}_n(f))}.$
\end{proof}

\begin{cor}\label{cor2} Let $(X,d)$ be a dendrite and $f:X\rightarrow X$ a monotone dendrite map. For each $n\in \mathbb{N}$, we have:
\begin{itemize}
\item[(a)] $\omega_{\mathcal{T}_n(f)}(T)$ is minimal for each $T\in \mathcal{T}_n(X)$.
\item[(b)] $\overline{\textrm{P}(\mathcal{T}_n(f))}=\Lambda(\mathcal{T}_n(f))=\textrm{RR}(\mathcal{T}_n(f))$.
\end{itemize}
\end{cor}
\begin{proof} (a) It follows immediately from Theorem \ref{prop2}.\\
(b) By Lemma \ref{lem17} and by Theorem \ref{prop2}, we have $\Lambda(\mathcal{T}_n(f))=\textrm{RR}(\mathcal{T}_n(f))\subset \overline{\textrm{P}(\mathcal{T}_n(f))}$. So it suffices to prove that $\overline{\textrm{P}(\mathcal{T}_n(f))}\subset \textrm{RR}(\mathcal{T}_n(f))$. Let $(T_n)_n$ be a sequences of periodic points of $\mathcal{T}_n(f)$ that converges to $T$. Clearly, $\textrm{End}(T_n)\subset \textrm{P}(f)$ for each $n\in\mathbb{N}$. We prove first that $\textrm{End}(T)\subset \overline{\textrm{P}(f)}$. Let $e\in \textrm{End}(T)$  and  $b\in \textrm{Cut}(T)$ such that $(b,e]$ is open in $T$. There is $\mu > 0$ such that $B(e,\mu)\cap T\subset (b,e]$. Let $0<\eps<\mu$ and $0<\delta<\eps$ be as in Lemma \ref{lem8}. Due to the continuity of the map $r_T$, we have $\underset{n\to +\infty} \lim d_H(r_T(T_n),T)=0$. Hence, for some $n \in \mathbb{N},\ d_H(r_T(T_n),T)<\delta$ and $d_H(T_n,T)<\delta$.
Thus, there exists $x_n\in T_n$ such that $r_T(x_n)\in B(e,\delta)$ and so $r_T(x_n)\in (b,e]$ and hence $d(r_T(x_n),e)< \delta$, and hence by Lemma \ref{lem8}, $\textrm{diam}([r_T(x_n),e])<\eps$. As $r_T(T_n)$ is a subtree of $T$, there is $y_n\in T_n$ for which $r_T(y_n)\in [r_T(x_n),e]$ and $(r_T(y_n),e]\cap r_T(T_n)=\emptyset$. Take $e_n\in \textrm{End}(T_n)$ such that $r_T(e_n)=r_T(y_n)$. As $d_H(T,T_n)<\delta$, there exists $x\in T$ such that $d(x,e_n)<\delta$, so $\textrm{diam}([x,e_n])<\eps$ and as $r_T(y_n)=r_T(e_n)\in[x,e_n]$, it follows that $d(e_n,r_T(y_n))<\eps$. So, $d(e_n,e)\leq d(e_n,r_T(y_n))+d(r_T(y_n),e)<2\eps$.
Consequently, $\textrm{End}(T)\subset \overline{\textrm{P}(f)}=\textrm{RR}(f)$. Now, we are going to prove that $T\in \textrm{RR}(\mathcal{T}_n(f))$. Suppose that $\textrm{End}(T)=\{e_1,...,e_k\}$ and $\eps>0$, let $0<\delta<\varepsilon$ be as in Lemma \ref{lem8}, there is $N\in \mathbb{N}$ for which $d(e_i,f^{kN}e_i)<\delta$ for each $i\in \{1,...,k\}$ and for any $k\in \mathbb{N}$. By assertion (b) of Lemma \ref{lem11}, $[f^{kN}(e_1),...,f^{kN}(e_k)]=f^{kN}(T)$ and by Lemma \ref{lem12}, $d_H(T,f^{kN}(T))<\varepsilon$ for any $k\in \mathbb{N}$. So, $T$ is regularly recurrent for $\mathcal{T}_n(f)$.

\end{proof}
The following Corollary follows immediately from Corollary \ref{cor2} and Theorem \ref{th2}.
\begin{cor}Let $f:X\rightarrow X$ be a monotone dendrite map. Then for any $A\in \mathcal{T}_n(X),\  \omega_{\mathcal{T}_n(f)}$ is either finite or a minimal Cantor set.
\end{cor}
We also deduce easily from Theorem \ref{prop2} the following:
\begin{cor} The induced map $\mathcal{T}_n(f)$ generated by a monotone dendrite map has no Li-Yorke pair. In particular, it has zero topological entropy.
\end{cor}
\section{\bf Example of chaotic induced map generated by a homeomorphism dendrite}
 Recall that in the case of $X$ being a tree and $f:X\rightarrow X$ a continuous map, Matviichuk proved in \cite{Matviichuk} that for any subtree $A$ of $X$, $\underset{n\to +\infty}\lim \mathrm{diam} f^n(A)=0$ (so the orbit of $A$ behaves similarly as the orbit of a point) or $A$ is asymptotically periodic (that is, its $\omega$-limit set is a periodic orbit). As a consequence, there is always equality between the topological entropy of $f$ and the topological entropy of its induced map $\mathcal{C}(f)$ (\cite{Matviichuk}, Theorem 4.3). In the case of $X$ being a dendrite it may happen that the dynamic of a continuous map $f:X\rightarrow X$ is quite simple but its induced map $\mathcal{C}(f)$ is quite complicated. For instance, it was shown in \cite{Acosta} that for some homeomorphism of a dendrite $f:X\to X $, its induced map $\mathcal{C}(f)$ has infinite topological entropy and then it is Li-Yorke chaotic but $f$ has no Li-Yorke pair  and with zero topological entropy. In this section, we give a similar example of a homeomorphism dendrite $g:S\rightarrow S$, where its induced map $\mathcal{C}(g)$  has infinite topological entropy and we show in addition that it is $\omega$-chaotic. (Notice that $\mathcal{C}(f)$ in the example of Section 7 in \cite{Acosta} is also $\omega$-chaotic but there is no proof of this fact in \cite{Acosta}).

\medskip
\underline{\textit{The dendrite $S$.}} We define the dendrite $S$ as a subset of the complex plane as follows: $S=\displaystyle\cup_{n\in \mathbb{Z}}I_n$ where for each $n\in \mathbb{Z_+}$, we let\\
 $I_{-n}=\{te^{\frac{i\pi}{n+2}},\ 0\leq t\leq \frac{1}{n+1}\}$ and $I_n=\{te^{i(\pi-\frac{\pi}{n+2})},  0\leq t\leq \frac{1}{n+1}\}$.\\

\underline{\textit{The map $g$.}} For any $n\in \mathbb{Z}_+$, we let
\begin{itemize}
  \item $g(te^{\frac{i\pi}{n+3)}})=\frac{n+2}{n+1}te^{\frac{i\pi}{n+2}},\ \forall t\in [0,\frac{1}{n+2}]$.
  \item $g(te^{i(\pi-\frac{\pi}{n+2})})=\frac{n+1}{n+2}te^{i(\pi-\frac{\pi}{n+3})}, \forall t\in [0,\frac{1}{n+1}]$.
\end{itemize}

In this way, $g$ is a homeomorphism of $S$ satisfying the following properties:
\begin{itemize}
\item[*]$\textrm{Fix}(g)=\{0\}$.
\item[*]$\forall n\in \mathbb{Z}$, $g(I_n)=I_{n+1}$.
\end{itemize}
So any point in $S$ is asymptotic to $0$. This illustrates the triviality of the dynamic of $f$.\ Moreover, since $S$ is a dendrite and $g$ is a homeomorphism, by (\cite{Acosta1}, Corollary 3.9), the topological entropy of $g$ is zero.

\underline{\textit{$\omega$-chaos for the induced map $\mathcal{C}(g)$}}: we are going to construct an uncountable $\omega$-scrambled set for $\mathcal{C}(g)$. Let $\{a_n,\ n\in \mathbb{N}\}$ be a dense sequence in $[0,1]$ and let for any $\lambda \in (0,1],\ h_{\lambda}:[0,1]\rightarrow [\frac{\lambda}{2},\lambda], t\mapsto \frac{\lambda}{2}t+\frac{\lambda}{2}$.
Let $a_{\lambda}(n)=h_{\lambda}(a(n)),\ \forall n\in \mathbb{N}$. So, for each $\lambda\in (0,1]$, we let $S_{\lambda}=\displaystyle\cup_{n\in \mathbb{N}}J_{-2^n}(\lambda)$ where
$$J_{-2^n}(\lambda)=\{te^{\frac{i\pi}{2^n+2}},\ 0\leq t\leq \frac{a_{\lambda}(n)}{2^n+1}\}\subset I_{-2^n}.$$
The set $\{S_{\lambda},\ \lambda\in (0,1]\}$ is an $\omega$-scrambled set: For any $\lambda\in (0,1],\ \{0\} \in \omega_{\mathcal{C}(g)}(S_{\lambda})$. Indeed, $d_H(g^{2^{n}+2^{n-1}}(S_{\lambda}),\{0\})\leq \frac{1}{2^{n-1}+1}\to 0$ when $n \to +\infty$.

For any $\alpha \in [\frac{\lambda}{2},\lambda]$, denote by $K_{\alpha}=[0,\alpha e^{\frac{i\pi}{2}}]$.
 Then we have $\{K_{\alpha}, \alpha\in [\frac{\lambda}{2},\lambda]\}\subset \omega_{\mathcal{C}(g)}(S_{\lambda})$: As $\{a_{\lambda}(n), n\in \mathbb{N}\}$ is a dense subset of $[\frac{\lambda}{2},\lambda]$, there is a sequence $(m_i)_i$ of positif integers such that $\underset{i\to+\infty}\lim m_i=+\infty$ and $\underset{i\to+\infty}\lim a_{\lambda}(m_i)=\alpha$. Hence, we get $d_H(K_\alpha,f^{2^{m_i}}(S_{\lambda}))\leq max\{\frac{1}{2^{m_i-1}+1},\mid a_\lambda(m_i)-\alpha\mid\}$.

 Let $\lambda, \lambda' \in (0,1]$ such that $\lambda< \lambda'$ then\\
 \begin{itemize}
   \item $\{K_{\alpha}, \lambda< \alpha\leq \lambda'\}\subset \omega_{\mathcal{C}(f)}(S_{\lambda'})\setminus [\omega_{\mathcal{C}(g)}(S_{\lambda})\cup P(\mathcal{C}(g))]$.
   \item $\{K_{\alpha}, \frac{\lambda}{2}< \alpha\leq min(\frac{\lambda'}{2},\lambda)\}\subset \omega_{\mathcal{C}(g)}(S_{\lambda})\setminus [\omega_{\mathcal{C}(g)}(S_{\lambda'})\cup P(\mathcal{C}(g))].$
 \end{itemize}

 Hence, $\{S_{\lambda},\ \lambda \in (0,1]\}$ is an $\omega$-scrambled set. Clearly, $\{S_{\lambda},\ \lambda \in (0,1]\}$ is uncountable. So, $\mathcal{C}(g)$ is $\omega$-chaotic.\\

\underline{\textit{Topological entropy for the induced map $\mathcal{C}(g)$:}}
Let $k\in \mathbb{N}$ and $\varepsilon=\frac{1}{k}$. For any $n\in \mathbb{N}$ and for any $\sigma=(\sigma_1,...,\sigma_n)\in \{1,2,...,k\}^n$, we let $T_{\sigma}$ be the subtree of $S$ defined as follow:
 $$T_{\sigma}=\displaystyle\cup_{j=1}^{n}[0,\frac{\sigma_j}{k(j+1)}e^{\frac{i\pi}{j+2}}]$$
For any $\sigma\neq \sigma'\in \{1,2,3,...,k\}^n$, there exists $j\in \{1,...,k\}$ such that $\sigma_j\neq\sigma_j'$. So, $d_H(g^j(T_{\sigma}),g^j(T_{\sigma'}))\geq \frac{1}{k}$. Thus, $\{T_{\sigma},\ \sigma\in \{1,2,...,k\}^n\}$ is an $(n,C(g),\varepsilon)$-separated set. It follows that $sep(n,C(g),\varepsilon)\geq k^n$. Hence, $h(C(g))\geq \ln(k), \forall k\in \mathbb{N}$. Therefore, $h(C(g))=+\infty$.


\begin{thebibliography}{9}
\bibitem{Acosta1}G. Acosta, P. Eslami, L.G. Oversteegen, \emph{On open maps between dendrites}, Houston J. Math. \textbf{33} (2007), 753-770.
\bibitem{Acosta2} G. Acosta, R. Hern\'andez-Guti\'errez, I. Naghmouchi, P. Oprocha, \emph{Periodic points and transitivity
on dendrites}, arXiv:1312.7426v1 (2013).
\bibitem{Acosta} G. Acosta, A. Illanes, H. M\'endez-Lango, \emph{The transitivity of induced maps}, Topology and its Applications, \textbf{156} (2009), 1013-1033.

\bibitem{Adler} R. Adler, A. Konheim, M.H. McAndrew, \emph{Topological entropy}, Trans. Amer. Math. Soc.  \textbf{114} (1965), 309-319.
\bibitem{Charatonic} D. Ar\'evalo, W.J. Charatonik, P. Pellicer-Covarrubias, L.C. Sim\'on, \emph{Dendrites with a closed set of endpoints}, Topology and its Applications, \textbf{115} (2001), 1-17.
\bibitem{Banks} J. Banks, \emph{Chaos for induced hyperspace maps}, Chaos, Solitons and Fractals, \textbf{25} (2005), 681-685.
\bibitem{Blanchard} F. Blanchard, E. Glasner, S. Kolyada, A. Maass, \emph{On Li-Yorke pairs}, J. Reine Angew. Math. \textbf{547} (2002), 51-68.
\bibitem{lsB} L.S. Block, W.A. Coppel, Dynamics in  One Dimension, Lecture Notes in Math, 1513. Springer-Verlag, 1992.
\bibitem{Block} L. Block, J. Keesling, \emph{A characterization of adding machine maps}, Topology and its Applications \textbf{140} (2004), 151-161.
\bibitem{Bowen}R. Bowen, \emph{Entropy for group endomorphisms and homogeneous spaces}, Trans. Amer. Math. Soc. \textbf{153} (1971), 401-414.
\bibitem{Dinaburg} E.I. Dinaburg, \emph{A connection between various entropy characterizations of dynamical systems}, Izv. Akad. Nauk SSSR Ser. Mat. \textbf{35} (1971), 324-366.
\bibitem{Mendez} L. G\'omez-Rueda, A. Illanes, H. M\'endez, \emph{Dynamic properties for the induced maps in the symmetric products}, Chaos, Solitons and Fractals,  \textbf{45} (2012), 1180-1187.

\bibitem{Nadler2} A. Illanes, S.B. Nadler, \emph{Hyperspaces: Fundamentals and Recent Advances}, Monogr. Textb. Pure Appl. Math., vol. 216, Marcel Dekker, New York, 1999.
\bibitem{Kura} K. Kuratowski, Topology, vol.2, Academic Press, New-York, 1986.
\bibitem{kocan} Z. Ko\~can, V. Kurkov\'a, M. M\'alek, \emph{On the centre and the set of $\omega$-limit points of continuous maps on dendrites}, Topology and its Applications, \textbf{156} (2009), 2923-2931.
\bibitem{Kwietniak} D. Kwietniak, P. Oprocha, \emph{Topological entropy and chaos for maps induced on hyperspaces}, Chaos, Solitons and Fractals, \textbf{33} (2007), 76-86.
\bibitem{Li} S. Li, \emph{$\omega$-chaos and topological entropy}, Trans. Amer. Math. Soc. \textbf{339} (1993), 243-249.
\bibitem{Ma2} J.H. Mai, E.H. Shi, \emph{$\overline{R}=\overline{P}$ for maps of dendrites $X$ with $\textrm{Card}(\textrm{End}(X))<c$}, Int. J. Bifurcation and Chaos, \textbf{19} (4) (2009), 1391-1396.
\bibitem{Marzougui} H. Marzougui, I. Naghmouchi, \emph{On totally periodic $\omega$-limit sets}, arXiv:1406.4401v2 (2014).

\bibitem{Matviichuk} M. Matviichuk, \emph{On the dynamics of subcontinua of a tree}, Journal of Difference Equations and Applications, \textbf{19}  (2) (2013), 223-233.
\bibitem{Mendez2} H. M\'endez, \emph{On Density of periodic points induced hyperspace maps}, Topology Proceeding, \textbf{35} (2010), 281-290.
\bibitem{Nadler} S.B. Nadler, Continuum Theory: An Introduction, (Monographs and Textbooks in Pure and Applied Mathematics, 158). Marcel Dekker, Inc., New York, 1992.
 \bibitem{Nagh} I. Naghmouchi, \emph{Dynamical properties of monotone dendrite maps}, Topology and its Applications, \textbf{159} (2012), no.\ 1, 144-149.
\bibitem{Nagh2} I. Naghmouchi, \emph{Pointwise recurrent dendrite maps}, Ergodic Theory and Dynamical Systems, \textbf{33} (2013), 1115-1123.

\bibitem{Pikula} R. Pikula, \emph{On some notions of chaos in dimension zero}, Colloq. Math. \textbf{107} (2007), 167-177.
\bibitem{Rees} M. Rees, \emph{A minimal positive entropoy homeomorphism of the 2-torus}, J. London Math. Soc. \textbf{23} (1981), 537-550.


\end{thebibliography}
\end{document}